\documentclass[11pt,twoside]{amsart}

\usepackage{enumerate}

\usepackage[hyperindex=true,plainpages=false,colorlinks=false,pdfpagelabels]{hyperref}

\theoremstyle{plain}
\newtheorem{The}{Theorem}[section]
\newtheorem*{The*}{Theorem}
\newtheorem{Pro}[The]{Proposition}
\newtheorem*{Pro*}{Proposition}
\newtheorem{Lem}[The]{Lemma}
\newtheorem{Cor}[The]{Corollary}
\newtheorem*{Cor*}{Corollary}

\theoremstyle{definition}
\newtheorem*{Def}{Definition}

\theoremstyle{remark}

\newtheorem*{Rem*}{Remark}

\numberwithin{equation}{section}
\DeclareMathOperator{\image}{im}
\DeclareMathOperator{\End}{End}

\DeclareMathOperator{\SO}{SO}

\DeclareMathOperator{\Id}{Id}

\DeclareMathOperator{\mH}{\mathcal H}
\DeclareMathOperator{\mV}{\mathcal V}
\DeclareMathOperator{\mJ}{\mathcal J}
\DeclareMathOperator{\ed}{d}
\DeclareMathOperator{\tr}{tr}
\DeclareMathOperator{\vol}{vol}
\DeclareMathOperator{\dive}{div}

\DeclareMathOperator{\grad}{grad}

\DeclareMathOperator{\del}{\partial}

\newcommand{\R}{\mathbb{R}}

\newcommand{\N}{\mathbb{N}}
\newcommand{\Z}{\mathbb{Z}}

\begin{document}

\title[Harmonic morphisms on conformally flat 3-spheres]{Harmonic morphisms on conformally flat 3-spheres}

\author{Sebastian Heller}

\address{Sebastian Heller\\
  Mathematisches Institut\\
  Universit{\"a}t T{\"u}bingen\\\
  Auf der Morgenstelle 10\\
  72076 T{\"u}bingen\\
  Germany}

\email{heller@mathematik.uni-tuebingen.de}


\date{\today}

\thanks{Author supported by GRK 870 ''Arithmetic and Geometry'' and SFB/Transregio 71}

\begin{abstract} 
 We show that under some non-degeneracy assumption the only submersive harmonic morphism on a conformally flat $3-$sphere is the Hopf fibration. The proof involves
 an appropriate use the Chern-Simons functional.
 \end{abstract}

\maketitle

\section{Introduction}
\label{sec:intro}
Harmonic morphisms have been the subject of intensive investigations \cite{BW}. They give rise to an overdetermined system of differential equations, namely they are harmonic and horizontally conformal. There exist many rigidity results, for example on space forms harmonic morphisms are classified, see
\cite{BW} for dimension $3$ and \cite{Br} in general.
Moreover in case of fibers of dimension one, there exists a local structure theorem for the metric on the domain. From this one can see that a Riemannian $3-$space of non-constant curvature possesses at most two different foliations which become harmonic morphisms locally, see \cite{BW2} and \cite{BW3}.

In the present article we give the answer to another question concerning harmonic morphisms: Which conformally flat metric on a $3-$sphere gives rise to a 
globally defined harmonic morphism?
We show, under a natural assumption, that the only conformally flat metric
on $S^3$ which posseses a submersive harmonic morphism onto a
surface is the round metric, and that the map must be the Hopf
fibration up to isometries. This result should be compared with a paper  of Pantilie \cite{P} where the case of 
harmonic morphisms on conformally flat domains of dimension $n\geq4$ is 
investigated. As the condition of being 
conformally flat on a space of dimension 
$n\geq4$ is totally different to the case of $3-$dimensional manifolds, the methods therein are different from ours.

The author would like to thank his thesis supervisor Ulrich Pinkall.
\section{Harmonic Morphisms on $3-$Manifolds}
We give a short introduction to harmonic morphisms. 
We refer the reader to the monograph \cite{BW} for a detailed study. We compute the curvature of the $3-$dimensional domain of a submersive harmonic morphism.
\subsection{Harmonic Mappings}
Let $f\colon (P,g)\to (M,h)$ be a smooth map between Riemannian manifolds.
The energy functional of $f$ is given by
$$E(f):=\frac{1}{2}\int_P\parallel\ed f\parallel^2\vol_P.$$
It is a generalization of the energy of real-valued functions and one defines harmonic maps as the critical values of this functional. Both, the functional and the Euler-Lagrange equation have the same shape as for functions: Consider the differential of $f$ as a section 
$$\ed f\in\Gamma(P;T^*P\otimes f^*TM),$$ 
and equip the bundle with the product connection of the Levi-Civita
connection on $T^*P$ and the pullback $f^*\nabla$ of the Levi-Civita
connection on $TM.$ Note that in case of functions, i.e.
$(M,h)=(\R,<,>),$ the product connection on 
$T^*P\otimes T\R\cong T^*P$ equals the Levi-Civita connection on
$T^*P.$
\begin{Pro*}
A map $f\colon (P,g)\to (M,h)$ between Riemannian manifolds is a harmonic map if and only if the tension field 
$$\tau(f):=\tr \nabla\ed f\in\Gamma(f^*TM)$$ vanishes.
\end{Pro*}
A proof can be found in \cite{EW} or \cite{J}. In the case of functions, the tension field
is given by the negative of the Laplacian so both definitions coincide.

There has been much research on harmonic maps. For example the problem of finding a harmonic map in the homotopy class of a
given map. But we are merely interested in a special class of
harmonic maps, which we study by using different methods than usually done for harmonic maps.

\subsection{Harmonic Morphisms}
A harmonic morphism is a map $$f\colon  (P,g)\to (M,h)$$ between Riemannian manifolds, such that for any locally defined harmonic
function $u\colon U\subset M\to\R$ the composition $u\circ f$ is
harmonic on $\pi^{-1}(U)\subset P.$ 

Of course, constant maps are harmonic morphisms. Further holomorphic maps between Riemannian surfaces or isometries are harmonic morphisms, too. The composition of two
harmonic morphisms is again a harmonic morphism. Therefore, in the case of a surface as target space, a map is a
harmonic morphism for a metric on the surface if and only if it is one for
any other metric in the same conformal class. This allows us to
speak about harmonic morphisms into Riemannian surfaces.\\

\begin{Def}
A submersion $\pi\colon P\to M$ between Riemannian manifolds is called conformal if for all $p\in P$
the differential $\ed_p\pi\colon\mH\to T_{\pi(p)}M$ restricted to the horizontal space  $\mH=\ker\ed\pi^{\perp}$ is conformal.
\end{Def}
With this we state the following useful characterization of harmonic morphisms given by Fuglede (\cite{Fu}) and Ishihara (\cite{Is}).

\begin{The*}\label{harhar}
A submersion between Riemannian manifolds is a harmonic morphism if and only if it is harmonic and conformal.
\end{The*}

This can be reformulated in a more appropriate way for our propose here. In case of
submersions from $3-$manifolds to surfaces the fibers are curves,
and the tension is given by (the projection of) the geodesic curvature.

\begin{The}[\cite{BE}]\label{hclass}
A submersion $\pi$ from a Riemannian $3-$space $(P,g)$ to a
Riemannian surface $(M,[h])$ is a harmonic morphism if and only if it
is conformal and has minimal fibers.
\end{The}
\subsection{Curvature of Fibered $3-$Manifolds}
Instead of working with the metric on $P$ for which $\pi$ is a harmonic morphism we change it by a conformal factor such that $\pi$ becomes into a special Riemannian submersion. We first collect some basic facts of Riemannian submersions,
for details see \cite{B}.

We will mainly discuss the case on hand: Let $\pi\colon (P,g)\to(M,h)$ be a Riemannian submersion between oriented Riemannian manifolds
of dimension $3$ and $2.$
The geometry of the total space of a Riemannian submersion is determined by the geometry of its fibers, of its base space and of the horizontal distribution $\mH:=\ker\ed\pi^{\perp}.$ Let $\nabla$ be the Levi-Civita connection on $P,$ and let 
$$T\in\Gamma(P;\mV)\subset\Gamma(P;TP)$$
be the unit length vector field in positive fiber direction. The geodesic curvature of the fiber is given by $\nabla_TT.$

To relate the geometry of the base to the one of the total space, we will use orthonormal vector fields 
$$A,\ B\in\Gamma(U;TM)$$ defined on an open subset $U\subset M$ together with a function $$\lambda\colon U\subset M\to\R$$ such that
$e^{\lambda}(A+iB)$ is a holomorphic vector field on the surface. This condition is equivalent to 
$[e^\lambda A,e^\lambda B]=0.$ Note that the Gaussian curvature of the surface is given by
$K=\Delta\lambda.$
We denote by $$\hat A,\hat B\in\Gamma(\pi^{-1}(U),\mathcal H)$$ the horizontal lifts of $A$ and $B,$ i.e. the unique horizontal vector fields which are $\pi-$related to $A$ and $B.$ 

The curvature of the horizontal distribution, i.e. the obstruction of the horizontal bundle being integrable, can be identified
with a real valued $2-$form 
$$\Omega\in\Omega^2(P);\ (X,Y)\mapsto -g([\pi^{\mH}(X),\pi^{\mH}(Y)],T),$$
where $\pi^{\mH}$ is the orthogonal projection onto the horizontal space 
$\mH.$ This form is horizontal, and because the base is of dimension $2$ we can use the metric and the orientation to define the curvature function $H$ of the horizontal distribution by the formula
\begin{equation}\label{defH}
\Omega=H\pi^*\vol_M,
\end{equation}
where $\vol_M$ is the volume form of the surface. In terms of the vector fields $\hat A$ and $\hat B$ the function $H$ is given by $H=-g([\hat A,\hat B],T).$

 Let $$\mJ\in\End(TP);\ X\mapsto T\times X$$ be the CR structure of the Riemannian submersion $\pi,$ i.e. 
$\mJ(\hat A)=\hat B,$
$\mJ(\hat B)=-\hat A$ and $\mJ(T)=0.$ With these notions the Levi-Civita connection is given by
\begin{Pro}\label{lc2}
The Levi-Civita connection on a Riemannian fibered $3-$space is given 
in terms of the vector fields $T,\hat A,\hat B$ by
\begin{equation*}
\begin{split}
\nabla T&=\frac{1}{2}H\mJ +\nabla_TT\otimes g(.,T)\\
\nabla\hat A&=\hat B\otimes g(.,\frac{1}{2}HT+\mJ\grad\lambda)
+T\otimes g(.,\frac{1}{2}H\hat B-g(\nabla_TT,\hat A)T)\\
\nabla\hat B&=-\hat A\otimes g(.,\frac{1}{2}HT+\mJ\grad\lambda)
-T\otimes g(.,\frac{1}{2}H\hat A+g(\nabla_TT,\hat B)T).\\
\end{split}
\end{equation*}
\end{Pro}
The computation above is a special case of the formulas of O'Neil for Riemannian submersions, see \cite{B}, and 
 just involves the use of the Kozul formula for the Levi-Civita connection.

It is well-known that in dimension $3$ the Riemannian curvature tensor $R$ is entirely given by the Ricci tensor.
For details of the decomposition of the Riemannian curvature tensor  see \cite{GHL}. For our purpose it is useful to work with the so-called Schouten tensor
$$S:=Ric-\frac{1}{4}scal \Id\in\End(TP)$$
instead of the Ricci tensor. For example,
$$R=-S\cdot g,$$
where $\cdot$ is the Kulkarni-Nomizu product, and where we consider all tensors to be bilinear or multilinear forms, respectively.

We only state the formulas for the Schouten tensor, which can be computed easily.

\begin{Pro}\label{s2}
The Schouten tensor of a Riemannian fibered $3-$manifold is given by
\begin{equation*}
\begin{split}
S(T,T)=&(-\frac{1}{2}K+\frac{5}{8}H^2+\frac{1}{2}\dive\nabla_TT),\\
S(T,X)=&g(\mJ(H\nabla_TT-\frac{1}{2}\grad^hH),X),\\
S(X,Y)=&(-\frac{3}{8}H^2+\frac{1}{2}K)g(X,Y)+\frac{1}{2}g(\nabla_X\mJ\nabla_TT,\mJ Y)\\
&+\frac{1}{2}g(\nabla_{\mJ X}\mJ\nabla_TT,Y)
+\frac{1}{2}g(\nabla_TT,\mJ X)g(\nabla_TT,\mJ Y)\\&-\frac{1}{2}g(\nabla_TT,X)g(\nabla_TT,Y),\\
\end{split}
\end{equation*}
where $X,Y$ are horizontal vectors, and $T,$ $H$ and $K$ are as above.
\end{Pro}
In case of a Riemannian submersion which is also a harmonic morphism,
the geodesic curvature of the fibers $\nabla_TT$ vanishes. Moreover the function $H$ will be constant fiber-wise. For a more detailed study see \cite{He2}.
The situation we are considering here is different: We only know that there exists a conformal equivalent metric such that $\pi$ is a harmonic morphism.
This also gives a constraint on $\nabla_TT.$ In the following we use the horizontal gradient
$$\grad^h:=\pi^{\mH}\circ\grad.$$
Then
\begin{Lem}\label{spme}
Let $\pi\colon (P^3,\tilde g)\to (M^2,[h])$ be a submersive harmonic
morphism such that its horizontal distribution has nowhere vanishing curvature. By a conformal change $g$ of the metric on $P$ and with a metric $h\in[h]$ on $M,$ $\pi$ is a Riemannian submersion, such that the mean curvature of the fibers is given by $$\nabla_TT=-\grad^h\log H$$ with respect to the new metric. For this metric, the function 
$$p\in M\mapsto\int_{\pi^{-1}(p)}\frac{1}{H}g(.,T)$$
is constant on the surface. We fix that constant to be $\pm\pi$ where the sign only depends on the orientations, then $g$ and $h$ are unique. 
\end{Lem}
\begin{proof}
Let $g$ and $h$ be metrics in the given conformal classes such that $\pi$ is a Riemannian submersion. Note that these metrics are unique up to the multiplication by the same function defined on $M.$ As the metric changes by the factor $e^{2\lambda}\colon M\to\R,$ the curvature function $H$ changes by the factor $e^{-\lambda}.$ By changing the orientation 
either on the total space or on the surface we can always assume $H>0.$ 
Hence there is an unique choice of the metrics $g$ and $h$ such that 
$$\int_{\pi^{-1}}\frac{1}{H}=\pi.$$

It remains to show that for this choice of $g$ and $h$ the geodesic curvature of the fibers is 
given by $\nabla_TT=-\grad^h\log H.$  In \ref{hclass} we have seen that by changing the metric $g$ in a suitable way the mean curvature of the fibers vanishes. Using the formula for the Levi-Civita connection for a conformal change of the metric we have that this is equivalent to the existence of a function $\lambda\colon P\to\R$ such that $\nabla_TT=\grad\lambda.$\\

Let $\hat A,\hat B$ be the horizontal lifts of positive oriented orthonormal basis fields 
$A,B$ on the surface. As $\hat A$ is a horizontal lift, the commutator $[\hat A,T]$ is vertical and  we get
$$[\hat A,T]=\nabla_{\hat A}T-\nabla_T\hat A=-g(\nabla_T\hat A,T)T
=g(\hat A,\nabla_TT)=\hat A\cdot\lambda T,$$
and similarly  $[\hat B,T]=\hat B\cdot\lambda T.$ By definition, we have 
$H=-g(T,[\hat A,\hat B]).$ The Jacobi identity for commutators of
the vector fields  $\hat A,\hat B, T$ yields  
$$T\cdot\lambda=-T\cdot\log H.$$
Hence, $\lambda=-\log H+f\circ\pi, $ where $f$ is some function defined on the surface. If we change the metric $g$ by the factor $e^{2\lambda},$ the fibers become geodesics, thus all of them
have the same length with respect to the new metric.
 We obtain
\begin{equation*}
\begin{split}
const&=\int_{\pi^{-1}}\tilde g(.,\tilde T)=\int_{\pi^{-1}}e^{\lambda} g(.,T)\\
&=\int_{\pi^{-1}}\frac{e^f}{H} g(.,T)=e^f\pi.
\end{split}
\end{equation*}
Thus $f$ must be constant and consequently 
$\nabla_TT=-\grad^h\log H.$
\end{proof}
Of course, with of $\nabla_TT=-\grad^h\log H$ the formulas for the Levi-Civita connection and the Schouten tensor simplify. We do not state them here, but we will use these formulas later, see chapter \ref{ls}.

\section{The Chern-Simons Invariant and Conformally Flat $3-$manifolds }\label{cs}
Our proof of the rigidity theorem \ref{uniquen} uses a
global invariant of conformal $3-$manifolds. In their paper, \cite{CS}, Chern and Simons
introduced a geometric invariant of connections. Their theory plays an important role in the topology and knot theory of $3-$manifolds, since Witten has shown its connection to the Jones polynomial. 
We will only work with the Chern-Simons functional for Levi-Civita
connections here, see \cite{Ch}. We shortly describe the geometric significance of the Chern-Simons functional in conformal geometry. There is a related obstruction to a metric to be conformally flat. We compute this for our case on hand.

\subsection{The Chern-Simons Functional}\label{csf}
We first recall some formulas for the Levi-Civita connection and the Riemannian curvature in terms of frames.
Consider a locally defined section $s\in\Gamma(U;O(P))$
of the orthonormal frame bundle $O(P)\to P$ of a Riemannian space $(P,g).$ The vector fields $X_i=s(e_i)$ define an orthonormal basis of $T_pM$ for all 
$p\in U\subset P.$ Consider the dual $1-$forms $\theta_i=g(.,X_i).$ Their collection $(\theta_i)$ is the canonical form $\theta_s$ along $s\in\Gamma(O(P)).$ Cartan's method of moving frames gives us a
skew symmetric matrix of $1-$forms $(\omega_{ij})$ defined by the formula
$$\ed\theta_j=-\sum_i\omega_{ij}\wedge\theta_i.$$
They are related to the Levi-Civita connection form 
$\omega\in\Omega^1(O(P),\mathfrak{so}(n))$ by
$$s^*\omega=(\omega_{ij})\in\Omega^1(U,\mathfrak{so}(n)).$$  
The $\omega_{ij}$ can be obtained from the covariant derivatives of 
the $X_j$ via
$$\nabla X_j=\sum_i\omega_{ij}\otimes X_i.$$
The $\mathfrak{so}(n)-$valued curvature $2-$form can be computed from 
$s^*\Omega=(\Omega_{ij})$ with
\begin{equation}\label{omegar}
\Omega_{ij}=\ed\omega_{ij}+\sum_k\omega_{ik}\wedge\omega_{kj}
=\frac{1}{2}\sum_{k,l} R(X_k,X_l,X_i,X_j)\theta^k\wedge\theta^l,
\end{equation}
where $R$ is the Riemannian curvature tensor.

Let $P$ always be a compact oriented $3-$manifold with trivial tangent bundle. 
We are going to use the bundle $\SO(P)\to P$
of oriented orthonormal frames instead of $O(P)\to P.$
\begin{The*}
Let $g$ be a metric on $P$ and 
$\omega\in\Omega^1(\SO(P),\mathfrak{so}(3))$ be the Levi-Civita
connection. For any section $s\colon P\to\SO(P)$ we set
\begin{equation}
\begin{split}
CS(P,g,s):=&\frac{1}{8\pi^2}\int_Ptr(-\frac{1}{2}s^{*}\omega\wedge s^{*}
\ed\omega-\frac{1}{3}s^{*}\omega\wedge s^{*}\omega\wedge s^{*}\omega)\\
=&\frac{1}{8\pi^2}\int_P\tr(-\frac{1}{2}s^{*}\omega\wedge s^{*}\Omega
+\frac{1}{6}s^{*}\omega\wedge s^{*}\omega\wedge s^{*}\omega).
\end{split}
\end{equation}
The functional only depends on the conformal class of $g$ and
on the homotopy type of $s.$ 
Consequently, the Chern-Simons functional
\begin{equation}
CS((P,[g])):=CS(P,g,s)\mod\Z\in\R/Z
\end{equation}
is a conformal invariant.
\end{The*}
A proof can be found in \cite{CS}. We will use this result to compute some useful integrals explicitly.

\subsection{Conformally Flat $3-$Manifolds}
A Riemannian manifold $(P,g)$ is conformally flat if there exists a local conformal diffeomorphism into the sphere equipped with its round metric around each point. This is equivalent to the existence of a locally defined function $\lambda$ such that $e^{2\lambda}g$ is flat, see
\cite {KP} or \cite{HJ} for more details.

In case of dimension $2,$ every metric is conformally flat. This is based on the fact that a metric together with an orientation give rise to an almost complex structure. For dimensional reasons this is in fact a complex structure. Therefore there exists holomorphic charts, which are of course conformal.

There exists metrics which are not conformally flat. For dimensions $n\geq4$ being conformally flat is equivalent to the vanishing of the Weyl tensor $W,$ which is the reminder in the general curvature decomposition $$R=-S\cdot g+W.$$ The condition in dimension $3$ is of a higher order: Consider the Schouten tensor $S\in\End(TP)$ as a $TP-$valued $1-$form on $P.$ The Levi-Civita connection on $P$ gives rise to the absolute exterior derivative 
$$d^\nabla\colon\Omega^k(P;TP)\to\Omega^{k+1}(P;TP).$$
Then a metric is conformally flat if and only if 
\begin{equation*}
d^\nabla S=0.
\end{equation*}
It turns out that this is exactly the Euler-Lagrange equation for the Chern-Simons functional:
\begin{The}[\cite{CS}]
The critical values $CS(P,[g])$ of the Chern-Simons functional are exactly
the conformally flat spaces $(P,[g]).$
\end{The}
It is possible to deduce from the obstruction
 $\ed^\nabla S=0$ a system of
differential equations in terms of the geometric quantities $H,$ $K,$ and 
$\nabla_TT$ of a Riemannian submersion.  We only state a formula for one part of this equations. For this we need the horizontal Laplacian
$$\Delta^h:=-\dive\circ\grad^h.$$
Then
\begin{Pro*}
A necessary condition for a Riemannian fibered $3-$manifold to be conformally flat is
\begin{equation}\label{laph}
0=\Delta^hH+g(\nabla_TT,\grad H)+2H(H^2-K+\parallel\nabla_TT\parallel^2)
+3H\dive\nabla_TT.
\end{equation}
In case
$\nabla_TT=-\grad^{h}\log H$ this equation turns into
\begin{equation}\label{laph2}
\begin{split}
0=&2\Delta^h H+2g(\grad^h H,\grad^h\log H)+H(H^2-K)\\
=&2H\Delta^h\log H+H(H^2-K).\\
\end{split}
\end{equation}
\end{Pro*}
\begin{proof}
One easily computes that the right hand side of \ref{laph} equals $2g(d^\nabla S(A\wedge B),T),$ which is a
component of the absolute exterior derivative of the Schouten tensor. It vanishes on conformally flat $3-$spaces.

The second equation follows easily by putting $\nabla_TT=-\grad^{h}\log H$ into \ref{laph}.
\end{proof}

\section{Harmonic Morphisms on Conformally Flat $3-$Spheres}\label{ls}
We will now study harmonic morphisms on a
 conformally flat $3-$sphere under two assumptions: We only consider submersive harmonic morphisms, and we restrict ourselves to the case where the curvature of
the horizontal distribution is nowhere vanishing. The latter is exactly the case where the induced CR structure on $S^3$ is strictly pseudo-convex.
Note that this is equivalent to the fact that the equation \ref{laph} is hypo-elliptic, see \cite{Hoe}.

By changing one of the orientations either on the $3-$space or on the surface ,the curvature function $H$ (\ref{defH}) of the horizontal distribution changes its sign.
So we will assume in the following that $H>0,$ and we say that the horizontal distribution is of positive curvature. Note that this is exactly the case where the submersion is homotopic to the Hopf fibration.

For the rest of the paper we are going to use the metrics given by 
\ref{spme}.
\begin{Pro}\label{intvol}
Let $\pi\colon (S^3,\tilde g)\to (S^2,[h])$ be a submersive harmonic
morphism such that its horizontal distribution has positive curvature.
Let $g$ on $S^3$ and $h$ on $S^2$ be given as in \ref{spme}
with corresponding volume forms $\vol_{S^3}$ and $\vol_{S^2}.$ Then we have
$$\int_{S^2}\vol_{S^2}=\pi,$$ and consequently
$$\int_{S^3}\frac{1}{H}\vol_{S^3}=\pi^2.$$
\end{Pro}
\begin{proof} Recall that $H>0.$ Set $\theta_3=g(.,T).$
Then the equation $\nabla_TT=-\grad^h\log H$ together with a short computation imply
$$\ed\frac{1}{H}\theta_3=\pi^*\vol_{S^2}.$$ Because of
$\int_{\pi^{-1}(p)}\frac{2}{H}\theta_3=2\pi,$ the flow of $\frac{H}{2}T$
gives rise to a principal $S^1-$bundle $\pi\colon S^3\to S^2,$
and one easily sees 
that $$\omega:=\frac{2}{H}\theta_3$$ is a principal connection form.
Then the curvature form of this connection is given by 
$$\Omega=2\vol_{S^2}.$$
But the degree of this bundle is $-1,$ and can be determined as
$$-2\pi\deg(S^3\to S^2)=\int_{S^2}\ed(\frac{2}{H}\theta_3)=\int_{S^2}2\vol_{S^2}.$$
Therefore $\int_{S^2}\vol_{S^2}=\pi,$ and 
$\int_{S^3}\frac{1}{H}\vol_{S^3}=\int_{S^2}\pi\vol_{S^2}=\pi^2$ by Fubini.
\end{proof}
We are going to compute the Chern-Simons functional in terms of the geometric quantities $H,\ K$ and $\nabla_TT=-\grad^h\log H.$ 
For this we need the following observation:
\begin{Lem}\label{updown}
Let $\pi\colon S^3\to S^2$ be a Riemannian submersion and $X\in\Gamma(S^3,\mH)$ be a horizontal vector field of length $1.$
Then the mapping degree of
$$p\in\pi^{-1}(q)\mapsto \ed_p\pi(X_p)\in S^1\subset T_qS^2$$
is $\pm 2$ for each fiber, where the sign is given by the sign of the degree of the bundle $\pi.$

Conversely, let $U\subset S^2$ be a nonempty open set such that $S^2\setminus U$ is simply connected 
with nonempty interior.
Let $A$ be a non-vanishing vector field of length $1$ defined on $U\subset S^2,$
and $e^{i\varphi}\colon\pi^{-1}(U)\to S^1$ be a map. Then
$$\cos\varphi\hat A+\sin\varphi\mJ\hat A$$ can be extended to a globally defined, non-vanishing, horizontal vector field of $S^3$ if and only if the mapping degree of  $e^{i\varphi}$ is $\pm 2$ 
for each fiber, with the same sign as above.
\end{Lem}
\begin{proof}
Every submersion of $S^3$ is homotopic to the Hopf fibration, which has
degree $-1,$ or to the conjugate Hopf fibration with degree $1.$
Both of them differ only by orientation, thus the proof of the lemma reduces to one of these cases.

The fibers of the conjugate Hopf fibration are given by the oriented integral curves of the left invariant vector field $I$, and for the left invariant horizontal field 
$J,$ the mapping degree of the projection is $2$
for each fiber.
Every other horizontal vector field is given
by $X=e^{i\varphi}J$ for some well-defined $e^{i\varphi}\colon S^3\to S^1.$ Further we have that $S^3$ is simply connected, hence the logarithm $\varphi\colon S^3\to\R$ is
well-defined, too, and consequently $X$ also has degree $2.$ 

The proof of the inverse direction follows easily by reversing the arguments.
\end{proof}
\begin{Pro}\label{chsi}
Let $g$ be the conformally flat metric on $S^3$ as in \ref{spme}. Let $T$ be the unique vertical vector field of length $1$ in positive fiber direction 
and let $X,Y$ be horizontal. Then the Chern-Simons functional with respect to the positive oriented orthonormal frame $s=(X,Y,T),$ is
$$8\pi^2CS(S^3,g,s)=\int_{S^3}4\frac{K}{H}-\frac{1}{2}H^3\vol_{S^3},$$ where $K$ is the Gaussian curvature of the surface and $H$ is the curvature function of the 
horizontal distribution. 
\end{Pro}
\begin{proof}
Let $A,B$ be a positive oriented orthonormal frame
on $V:=S^2\setminus\{p\}$ for some $p\in S^2,$ and let $\lambda\colon V\to\R$ be a function such that 
$e^{\lambda}(A-iB)$ is holomorphic. We denote their horizontal lifts by 
$A$ and $B,$ too. For any frame $s=(X,Y,T),$ there is a function $e^{i\varphi}\colon\pi^{-1}(V)\to\R$ with $X+iY=e^{i\varphi}(A+iB).$ 
By using \ref{lc2}, we compute the following connection forms for the frame $s=(X,Y,T):$
\begin{equation}
\begin{split}
\omega^1=&\mJ^*\ed\lambda-\frac{1}{2}H\theta_3-\ed\varphi,\\
\omega^2=&-\frac{1}{2}H\theta_2+g(\nabla_TT,X)\theta_3,\\
\omega^3=&\frac{1}{2}H\theta_1+g(\nabla_TT,Y)\theta_3,
\end{split}
\end{equation}
where 
$$\theta_1=\cos\varphi\ g(.,A)+\sin\varphi\ g(.,B),\ 
\theta_2=-\sin\varphi\ g(.,A)+\cos\varphi\ g(.,B),\ \theta_3=g(.,T)$$ is the dual frame of 
$X,Y,T.$ With the arguments used in \ref{csf} and $\nabla_TT=-\grad^h\log H,$ we compute the integrand $\mu$ of the Chern-Simons functional on $\pi^{-1}(V):$
\begin{equation*}
\begin{split}
\mu=&(\frac{1}{2}H(H^2-K)+T\cdot\varphi(\frac{1}{2}H^2-K)+\frac{1}{2}\Delta^hH\\
&+g(\grad H,\mJ\grad\varphi-\grad\lambda+2\grad^h\log H)\vol_{S^3}.\\
\end{split}
\end{equation*}
A necessary condition (\ref{laph2}) for $g$ to be conformally flat is
$0=2H\Delta^h\log H+H(H^2-K).$
Thus we obtain
\begin{equation*}
\begin{split}
\mu=&(-\frac{1}{2}H(H^2-K)+T\cdot\varphi(\frac{1}{2}H^2-K)\\
&-\frac{3}{2}\Delta^hH+g(\grad H,\mJ\grad\varphi-\grad\lambda)\vol_{S^3}.\\
\end{split}
\end{equation*}
For $n\in\N$ we set $U_n:=\pi^{-1}(B_n),$ where the $B_n\subset B_{n+1}\subset V=S^2\setminus\{p\}$ are open, connected, and simply connected subsets with
piecewise smooth boundary, such that $\lim_{n\to\infty}\int_{S^2\setminus B_n}\vol_{S^2}=0$ and $\bar B_n\neq S^2.$ 
These sets will be specified later on.

Let $\eta\in\Omega^1(V)$ be a $1-$form on the surface such that $\ed\eta=4\vol_{S^2}.$ We also denote the
pull-back $\pi^*\eta$ by $\eta.$
We claim that there exists a map $$e^{i\varphi}\colon\pi^{-1}(V)\to S^1$$ with
$\ed\varphi=\eta-\frac{4}{H}\theta_3.$ To prove the existence of $e^{i\varphi},$ note that
$\ed(\eta-\frac{4}{H}\theta_3)=0$ by construction of $\eta,$ and that 
$$\int_{\pi^{-1}(q)}\eta-\frac{4}{H}\theta_3=-4\pi\in\ker(t\mapsto e^{it})$$ for 
any generator 
$\pi^{-1}(q),$ $q\neq p,$ of the first fundamental group of $\pi^{-1}(V).$ 
We have shown in \ref{updown} that there exists 
for each $U_n$ a globally defined orthonormal frame $s_n=s=(X,Y,T)$
with $X=\cos\varphi A+\sin\varphi B$ on $U_n.$
Thus
\begin{equation}\label{cs3}
\begin{split}
8\pi^2CS(S^3,g,s)=&\int_{S^3}\mu=lim_{n\to\infty}\int_{U_n}\mu\\
=\lim_{n\to\infty}&\int_{U_n}-\frac{1}{2}H(H^2-K)-\frac{4}{H}(\frac{1}{2}H^2-K)+\frac{1}{2}H(4-K)\vol_{S^3}\\
&+\lim_{n\to\infty}\frac{1}{2}\int_{\del U_n}H*(-\mJ^*\eta-\ed\lambda)\\
=\lim_{n\to\infty}&\int_{U_n}(\frac{4K}{H}-\frac{1}{2}H^3)\vol_{S^3}
+\lim_{n\to\infty}\frac{1}{2}\int_{\del U_n}H*(-\mJ^*\eta-\ed\lambda)\\
=&\int_{S^3}(\frac{4K}{H}-\frac{1}{2}H^3)\vol_{S^3}
+\lim_{n\to\infty}\frac{1}{2}\int_{\del U_n}H*(-\mJ^*\eta-\ed\lambda),\\
\end{split}
\end{equation}
where the third equality follows from
$$\delta(\mJ^*\eta+\ed\lambda)=K-4-g(\ed\log H,\mJ^*\eta+\ed\lambda)$$
and the following application of Stokes for each $n\in\N:$
\begin{equation*}\label{cs3}
\begin{split}
&\int_{U_n}g(\grad H,\mJ\grad\varphi-\grad\lambda)\vol\\
&=\int_{U_n}H\delta(-\mJ^*\eta-\ed\lambda)\vol\\
&+ \int_{\del U_n}H*(-\mJ^*\eta-\ed\lambda)\\
&=\int_{U_n}H(4-K)+g(\grad H,\grad\lambda-\mJ\grad\varphi)\vol.\\
&+ \int_{\del U_n}H*(-\mJ^*\eta-\ed\lambda).\\
\end{split}
\end{equation*}
It remains to show that 
$\lim_{n\to 0}\int_{\del U_{n}}H*(-\mJ^*\eta-\ed\lambda)=0.$
Let the function $h\colon S^2\to\R$ be defined by
$$q\in S^2\mapsto\int_{\pi^{-1}(q)}H\theta_3.$$
Note that for any form $\alpha\in\Omega^1(S^2)$ we have
$*\pi^*\alpha=\pi^**\alpha\wedge\theta_3.$
We apply Fubini to obtain
\begin{equation}\label{boundint}
\begin{split}
\int_{\del U_n}H*(-\mJ^*\eta-\ed\lambda)=&
\int_{\pi^{-1}(\del B_n)}H*(-\mJ^*\eta-\ed\lambda)\\
&=-\int_{\del B_n}h(\eta+*\ed\lambda),\\
\end{split}
\end{equation}
where the Hodge star $*$ in the last line is the one on the surface. On $V\subset S^2$  it is $$\ed(\eta+*\ed\lambda)=(4-K)\vol_{S^2}.$$
By using \ref{intvol} and Gauss-Bonnet we have  $$\int_{S^2}(4-K)\vol_{S^2}=0,$$ 
thus  $(4-K)\vol_{S^2}$ is an exact differential form on $S^2$ by Hodge theory, i.e. there exists $$\alpha\in\Omega^1(S^2)$$ with $\ed\alpha=(4-K)\vol_{S^2}.$ Therefore, on $V$ we have
$$\eta+*\ed\lambda=\alpha+\beta,$$ where $\beta$ is an appropriate closed $1-$form on $V.$ Since $V$ is simply connected we have 
$\beta=\ed f$ for some function $f\colon V\to\R.$ Because of Stokes
theorem, and because $\alpha$ and $h$ are defined on the whole of $S^2,$ we obtain
\begin{equation}\label{intalpha}
\begin{split}
\lim_{n\to\infty}\int_{\del B_n}h\alpha=
-&\lim_{n\to\infty}\int_{S^2\setminus B_n}\ed(h\alpha)=0.\\
\end{split}
\end{equation}
The only term left in (\ref{boundint}) to investigate is $\lim_{n\to\infty}\int_{\del B_n}h\ed f.$
We consider two cases:
In the first the function $h$ has an isolated critical point or is constant on an open set. Then, there are closed embedded curves $\gamma_n\colon S^1\to S^2$ 
and a point $p\in S^2,$ such that $h$ is constant along each 
$\gamma_n,$ and such that $\lim_{n\to\infty}\vol(B_n)=\vol_{S^2},$ where $B_n$ is the
component of $S^2\setminus\image\gamma_n$ with $p\notin B_n.$ It follows immediately that
$$\lim_{n\to\infty}\int_{\del B_n}h\ed f=\pm\lim_{n\to\infty}
\int_{\gamma_n}h\ed f=0.$$\\

In the other case, there is a point $p\in S^2$ with $\grad_p h\neq0.$ Furthermore, there exists a small neighborhood $\tilde V$ around $p$ such that $f$ is bounded on $S^2\setminus\tilde V$ and for each $c\in\R$ the set
$h^{-1}(c)\cap\tilde V$ is empty or consist of the image of a smooth, connected and open curve. In fact, one can choose $\tilde V$ together and diffeomorphism 
$$x\colon\tilde V\to I\times]h(p)-\epsilon;h(p)+\epsilon[$$ 
for an open interval $I=]a;b[,$
such that $h\circ x^{-1}(I\times\{c\})=\{c\}$ for all $c\in]h(p)-\epsilon;h(p)+\epsilon[.$ We may assume that  $x$ can be extend to a diffeomorphism defined on an open neighborhood of the closure
of $\tilde V.$ Let $$B_n:=S^2\setminus x^{-1}([a,b]\times[h(p)-\frac{1}{n};h(p)+\frac{1}{n}]).$$ Then $\lim_{n\to\infty}\vol(B_n)=\vol(S^2)$ and
its boundary 
$\gamma_n$ is a piecewise smooth, oriented and closed curve which consists of the points
$$\image\gamma_n=\del x^{-1}([a,b]\times[h(p)-\frac{1}{n};h(p)+\frac{1}{n}]).$$

With $c_n:=h(p)-\frac{1}{n}$ and $d_n:=h(p)+\frac{1}{n}$ we get 
\begin{equation*}
\begin{split}
\lim_{n\to\infty}\int_{\del B_n}h\ed f=&-\lim_{n\to\infty}\int_{\gamma_n}h\ed f\\
=&-\lim_{n\to\infty}(c_n(f(x^{-1}(b,c_n))-f(x^{-1}(a,c_n)))\\
&+ d_n(f(x^{-1}(b,d_n))-f(x^{-1}(a,d_n))))\\
&-\lim_{n\to\infty}\int_{x^{-1}(\{b\}\times[c_n;d_n])}h\ed f\\
&+\lim_{n\to\infty}\int_{x^{-1}(\{a\}\times[c_n;d_n])}h\ed f=0\\
\end{split}
\end{equation*}
because $f$ is bounded on $S^2\subset\tilde V.$
\end{proof}
\begin{Cor}\label{inth3}
Let $g$ be the conformally flat metric on $S^3$ given by \ref{spme}. Then
$$\int_{S^3}H^3\vol=16\pi^2.$$
\end{Cor}
\begin{proof}
As $H>0,$ $\pi$ is homotopic to the Hopf fibration, and the induced frames are homotopic with the same Chern-Simons functional. But in case of the Hopf fibration it is easy to compute the Chern-Simons 
functional for the induced section $\tilde s=(\bar J,\bar K,\bar I)\in\SO(S^3):$
One easily checks $\tilde H=2,$ $\tilde K=4$ and $\int_{S^3}\vol_{S^3}=2\pi^2,$ thus
$$CS(S^3,g_{round},\tilde s)=1.$$
With \ref{chsi} and
$$\int_{S^3}\frac{K}{H}\vol=\int_{S^2}\pi K\vol_{S^2}=4\pi^2,$$
one obtains 
$$\int_{S^3}H^3\vol=16\pi^2.$$
\end{proof}
We are now able to classify submersive harmonic morphisms from a conformally flat $S^3$ with nowhere vanishing horizontal curvature.
\begin{The}\label{uniquen}
Let $\pi\colon (S^3,g)\to (S^2,h)$ be a submersive harmonic morphism
of a conformally flat $(S^3,g).$ Assume that the curvature of the horizontal distribution is nowhere vanishing. Then $g$ is the round metric and $\pi$ is, up to isometries of $S^3,$ the Hopf fibration.
\end{The}
\begin{proof}
Again we can assume that $H>0.$
As the metric $g$ on $S^3$ is conformally flat, we can use \ref{laph2}: 
$$0=2H\Delta^h\log H+H(H^2-K).$$ By dividing this equation by $H^2$ and then integrating it, we obtain
\begin{equation*}
\begin{split}
0=&\int_{S^3}\frac{1}{H^2}(2H\Delta^h\log H+H(H^2-K))\vol_{S^3}\\
=&\int_{S^3}(\frac{2}{H}\Delta^h\log H+H-\frac{K}{H})\vol_{S^3}\\
=&-\int_{S^3}\frac{2}{H^3}\parallel\grad^hH\parallel^2\vol_{S^3}+\int_{S^3}H\vol_{S^3}-4\pi^2.
\end{split}
\end{equation*}
This shows that $$\int_{S^3}H\vol \geq 4\pi^2$$ with equality if and only if
$\grad^hH=0.$\\

Consider the measure $\mu=\frac{1}{H}\vol$ on $S^3.$ Then the Cauchy-Schwartz inequality, 
\ref{intvol} and \ref{inth3} give us
\begin{equation*}
\begin{split}
\int_{S^3}H\vol=&\int_{S^3}H^2\mu\leq(\int_{S^3}H^4\mu)^{\frac{1}{2}}
(\int_{S^3}1^2\mu)^{\frac{1}{2}}\\
=&(\int_{S^3}H^3\vol)^{\frac{1}{2}}
(\int_{S^3}\frac{1}{H}\vol)^{\frac{1}{2}}=\sqrt{16\pi^2}\sqrt{\pi^2}=4\pi^2.
\end{split}
\end{equation*}
Therefore, $\int_{S^3}H=4\pi^2$ and $\grad^hH=0.$ If $\grad_p H\neq0$ for
a point $p\in S^3,$ the level sets of $H$ near $p$ would be integral curves of the horizontal distribution $\mH$ which is a contradiction to the non-integrability of the horizontal distribution $\mH$ measured by $H>0.$ Thus, 
$H$ is constant and $\nabla_TT=0.$ Using \ref{laph2}, \ref{intvol}
and \ref{s2} one obtains that $g$ is of constant curvature and the fibers are circles.
The theorem then follows
from \cite{BW} or \cite{He3}.
\end{proof}
We have the following Corollary
\begin{Cor}
Let $P=L(d;1),$ $d\neq0,$ be a lens space and $\pi\colon P\to S^2$ be a submersion. Let $g$ be
a conformally flat metric on $P$ such that $\pi$ is a harmonic morphism. Assume that the horizontal distribution has nowhere vanishing curvature. Then $g$ is of constant curvature, and $\pi$ factorizes the Hopf fibration
(up to isometries).
\end{Cor}
\begin{proof}
The universal covering of $P$ is $S^3.$ Therefore theorem \ref{uniquen} shows that
the composition of the covering map and $\pi$ is, up to isometries, the Hopf fibration.
\end{proof}

\end{document}